\documentclass[11pt,letterpaper,reqno]{amsart}
\usepackage{tikz}
\usetikzlibrary{positioning, shapes.geometric, arrows.meta}
\usepackage{amssymb}
\usepackage{amsmath}
\usepackage{amsthm}
\usepackage{amsfonts}
\usepackage{bbm}
\usepackage{enumitem} 
\usepackage{pgfplots}
\pgfplotsset{compat=1.18} 
\usepackage{booktabs}

\usepackage{graphicx}
\usepackage[T1]{fontenc}
\usepackage{doi}
\usepackage{float} 
\usepackage{tikz}
\usepackage{pgfplots}
\pgfplotsset{compat=1.18}
\usetikzlibrary{arrows.meta}

\usepackage{bookmark}
\usepackage{hyperref}
\hypersetup{pdfstartview={FitH}}

\newtheorem{thm}{Theorem}[section]
\newtheorem{conj}[thm]{Conjecture}
\newtheorem{lem}[thm]{Lemma}
\newtheorem{prop}[thm]{Proposition}
\newtheorem{cor}[thm]{Corollary}
\theoremstyle{definition}
\newtheorem{rem}[thm]{Remark}
\newtheorem{defin}[thm]{Definition}
\numberwithin{equation}{section}

\newcommand{\D}{\mathbb D}
\newcommand{\T}{\mathbb T}
\newcommand{\C}{\mathbb C}
\newcommand{\dd}{\,dm}
\newcommand{\ip}[2]{\left\langle #1,#2\right\rangle}
\newcommand{\norm}[1]{\left\|#1\right\|}

\begin{document}

\title[Agler--McCarthy entropy conjecture]
{Proof of the Agler--McCarthy entropy conjecture}

\author[J.~Lei]{Jialin Lei}
\address{Department of Mathematics, Binghamton University, Binghamton, NY 13902, USA}

\email{jlei15@binghamton.edu}

\author[T.~Zhang]{Teng Zhang}

\address{School of Mathematics and Statistics, Xi'an Jiaotong University, Xi'an 710049, P. R. China}
\email{teng.zhang@stu.xjtu.edu.cn}



\begin{abstract}
In 2021, J.~Agler and J.~E. McCarthy proposed a two-step programme toward the celebrated
Krzy\.z conjecture.  The first step is to prove an entropy conjecture for
polynomials whose zeros all lie on the unit circle; the second is to establish a
full degree condition for extremal functions in the Krzy\.z conjecture.  The purpose
of this paper is to complete the first step.   More precisely,  we establish
the sharp homogeneous entropy inequality for all non-constant polynomials
with zeros on the unit circle and determine the equality cases.
\end{abstract}

\maketitle
\tableofcontents

\section{Introduction}
\subsection{Krzy\.z conjecture}
Let $\mathbb{D}=\{z\in\mathbb C: |z|<1\}$ denote the open unit disk. We use $H(\D)$ to denote the set of all holomorphic functions on \(\D\). Let
\[
\Omega=\{f\in H(\D): f(\D)\subset \D\setminus\{0\}\}.
\]
For a holomorphic function \(f\), write
\[
        f(z)=\sum_{k\ge0}\widehat f(k)z^k,
\]
where \(\widehat f(k)\) denotes its \(k\)-th Taylor coefficient at the origin.

In 1968, Krzy\.z~\cite{Krz68} proposed the following conjecture, now known as the \emph{Krzy\.z conjecture}.

\begin{conj}[Krzy\.z]\label{conj:Krzyz}
Let
$
   f\in \Omega$.
Then for every $n\geq 1$,
\[
K_n^\bullet
=
\sup\{|\widehat f(n)|: f\in\Omega\}= \frac{2}{e}.
\]
Moreover, equality holds for a given $n$ if and only if $f$ is of the form
\begin{equation*}
        f(z)
    =
    \alpha
    \exp\!\left(
        \frac{\zeta z^n-1}{\zeta z^n+1}
    \right),
    \qquad |\alpha|=|\zeta|=1 .
\end{equation*}
\end{conj}
Let us briefly recall some related developments of
Conjecture~\ref{conj:Krzyz}.

We first mention the known results for low order coefficients.  The case
$n=1$ is elementary. Krzy\.z
\cite{Krz68} proved Conjecture~\ref{conj:Krzyz} for $n=2$.  The case
$n=3$ was proved by Hummel, Scheinberg and Zalcman \cite{HSZ77}.
The case $n=4$ was proved first by Tan \cite{Tan83}, and later by a
different method by Brown \cite{Bro87}.  The case $n=5$ was proved by
Samaris \cite{Sam03}.  The cases $n\ge 6$ remain open.

We next recall non-sharp upper bounds which hold uniformly for all
coefficients.  Horowitz \cite{Hor78} proved that, for every $n\ge1$,
\[
 K_n^\bullet
    \le
    1-\frac{1}{3\pi}
      +\frac{4}{\pi}\sin\frac{1}{12}
    =
    0.99987\ldots .
\]
This was later improved slightly by Ermers~\cite{Erm90} to
$
  K_n^\bullet\le 0.9991\ldots .
$
Both estimates are still far from the conjectured sharp bound
$
   2/e=0.73575888\ldots .
$

We also mention several more recent structural developments. 
Mart\'in, Sawyer, Uriarte-Tuero and Vukoti\'c~\cite{MSUV15} proved that
sixteen different conditions are equivalent to the Krzy\.z conjecture.
Their paper also contains a useful historical summary and an illustration of the importance of the Krzy\.z conjecture.  Brevig, Grepstad and Instanes~\cite{BGI23} clarified the role of Wiener's
trick in related Hardy space extremal problems.

\subsection{The Agler--McCarthy programme}\label{subsec:AM}

Let
\[
\T=\{z\in\C: |z|=1\},\qquad dm(e^{i\theta})=\frac{d\theta}{2\pi}.
\]
Agler and McCarthy~\cite{AM21} discovered a striking connection between the
Krzy\.z conjecture and the following entropy conjecture for polynomials whose zeros
lie on the unit circle.  

\begin{conj}[Agler--McCarthy]\label{conj:AM}
Let \(p\) be a non-constant polynomial with all zeros on \(\T\), normalized by
$
\int_{\T}|p|^2\dd=1.
$
Then
\begin{equation}\label{eq:AM-conj-intro}
\int_{\T}|p|^2\log |p|^2\dd\ge 1-\log 2.
\end{equation}
Moreover, equality in \eqref{eq:AM-conj-intro} holds if and only if
\begin{equation}\label{eq:equality-intro}
p(z)=\frac{\zeta}{\sqrt2}\bigl(\omega+z^n\bigr),
\end{equation}
where \(n=\deg p\) and \(|\zeta|=|\omega|=1\).
\end{conj}

We now recall precisely how Conjecture~\ref{conj:AM} enters Conjecture~\ref{conj:Krzyz} in \cite{AM21}.

Since every \(f\in\Omega\) is zero free in the simply connected domain \(\D\), it
admits a holomorphic logarithm.  Thus we may write
\[
f=e^{-g}.
\]
Since \(|f|<1\) in \(\D\), this function \(g\) satisfies
$
\operatorname{Re}g>0.
$
Thus \(g\) belongs to the Herglotz class
\[
\mathcal H=\{g\in H(\D): \operatorname{Re}g\ge0\}.
\]
The ambiguity in the choice of the logarithm is only an additive purely imaginary
constant; this is harmless in what follows.

For \(m\ge1\), let \(\mathcal R_m\) denote the class of rational Herglotz
functions with exactly \(m\) boundary atoms:
\[
\mathcal R_m
=
\left\{
ia+\sum_{\ell=1}^m w_\ell
\frac{\tau_\ell+z}{\tau_\ell-z}:
a\in\mathbb R,\quad
w_\ell>0,\quad
\tau_\ell\in\T\ \text{distinct}
\right\}.
\]
Equivalently, elements of \(\mathcal R_m\) are the finite atomic Herglotz
functions whose representing measure has exactly \(m\) positive atoms on
\(\T\).  Put
\[
\mathcal R_n^\bullet
=
\bigcup_{1\le m\le n}\mathcal R_m.
\]

Agler and McCarthy \cite[Lemma~3.3]{AM21} use Pick's theorem to show that
if \(f\in\Omega\) is \(K_n^\bullet\)-extremal and \(f=e^{-g}\), then 
\begin{equation*}
g=-\log f\in\mathcal R_n^\bullet.
\end{equation*}
That is,
\[
g\in\mathcal R_m
\qquad\text{for some }1\le m\le n.
\]

The Agler--McCarthy programme for proving the Krzy\.z conjecture \cite[Theorem~1.8]{AM21} has two
separate steps:
\[
\boxed{
\text{Step I: prove the entropy conjecture, Conjecture~\ref{conj:AM};}
}
\]
\[
\boxed{
\text{Step II: establish the full degree condition }-\log f\in\mathcal R_n
\text{ for Krzy\.z extremals.}
}
\]
\subsection{Main result}
The purpose of this paper is to complete Step I: prove Conjecture~\ref{conj:AM}.  We shall in fact prove the following homogeneous form.

\begin{thm}\label{thm:main}
Let $p$ be a non-constant polynomial of degree $n\ge 1$, and suppose that all zeros of $p$ lie on $\T$.  Put
$
N(p)=\int_{\T}|p|^2\dd.
$
Then
\begin{equation}\label{eq:main-homogeneous}
\int_{\T}|p|^2\log |p|^2\dd
\ge
N(p)\left(1+\log\frac{N(p)}2\right).
\end{equation}
Equality holds if and only if
\begin{equation*}
p(z)=c\bigl(\omega+z^n\bigr),
\end{equation*}
where   $c\neq 0$ and $|\omega|=1$.  In particular, if $N(p)=1$, then $|c|=1/\sqrt2$, and \eqref{eq:AM-conj-intro} follows with equality precisely in the form \eqref{eq:equality-intro}.
\end{thm}

Our proof is based on a decomposition which appears naturally from the polar derivative.  If $D=z\frac{d}{dz}$ and $p$ is normalized to be self-inversive, then
$
q=p-\frac1nDp
$
satisfies
$
p=q+q^*,
$
where the star operation is always taken relative to the fixed degree $n$.  The crucial feature is that this particular $q$ is automatically zero free in $\D$.  Thus, for polynomials with simple zeros on $\T$, the quotient
$
r=q^*/q
$
is a finite Blaschke product with $r(0)=0$.

The entropy in  \eqref{eq:main-homogeneous} then splits as
\[
\int_{\T}|p|^2\log |p|^2\dd
=
\int_{\T}|p|^2\log |q|^2\dd
+
\int_{\T}|p|^2\log\left|\frac pq\right|^2\dd.
\]
The first term is bounded below by Jensen's inequality.  The second term is governed by a sharp polar derivative inequality:
\begin{equation}\label{eq:sharp-intro}
\int_{\T}|p|^2\log\left|\frac{p}{p-\frac1nDp}\right|^2\dd
\ge
\int_{\T}|p|^2\dd.
\end{equation}
The proof of \eqref{eq:sharp-intro} is the heart of the argument.  It uses the Fourier expansion of
\[
h(e^{it})=|1+e^{it}|^2\log |1+e^{it}|^2
\]
and a weighted contraction principle for multiplication by Schur functions.  We also give the details needed to justify the Fourier substitution and the passage from simple zeros to multiple zeros.
Here, we use the convention that $x\log x=0$ at $x=0$.  All logarithmic integrals in this paper are interpreted in this standard limiting sense.  When two nonzero polynomials have common boundary zeros, expressions containing the logarithm of their quotient are interpreted as differences of the corresponding logarithmic integrals; see Definition \ref{def:ratio-functional}.

\subsection*{Organization of the paper}

The paper is organized as follows.  Section~\ref{s:2} develops the self-inversive
normalization and introduces the polar factor
\(q=p-\frac1nDp\), leading to the decomposition \(p=q+q^*\).  We also prove
that \(q\) is zero free in \(\D\), and that, in the simple zero case,
\(q^*/q\) is a finite Blaschke product vanishing at the origin.
Section~\ref{s:3} proves the lower bound for the Jensen term.  Section~\ref{s:4} proves the
sharp polar derivative estimate, using the Fourier expansion of
\(|1+w|^2\log |1+w|^2\) and a weighted Schur contraction principle.  In
Section~\ref{s:5} we remove the assumption that the zeros are simple, by means of a
self-inversive approximation and a logarithmic continuity argument.  Section~\ref{s:6}
combines the preceding estimates to prove Theorem~\ref{thm:main} and to
determine the equality cases.

\section{The self-inversive normalization and the polar factor}\label{s:2}

For a complex polynomial $f$ of degree at most $n$, define its reflected polynomial, relative to the fixed degree $n$, by
\begin{equation*}
f^*(z)=z^n\overline{f(1/\bar z)}.
\end{equation*}
Thus, if $f(z)=\sum_{j=0}^n b_jz^j$, then
\[
f^*(z)=\sum_{j=0}^n\overline{b_{n-j}}z^j.
\]
A polynomial is called self-inversive, relative to degree $n$, if $f=f^*$. Recall that a finite Blaschke product is a rational function of the form
\[
        B(z)=\gamma\prod_{j=1}^m
        \frac{z-\alpha_j}{1-\overline{\alpha_j}z},
        \qquad |\gamma|=1,\quad \alpha_j\in\D .
\]

\begin{lem}\label{lem:self-inversive-normalization}
Let $p$ be a polynomial of degree $n$ whose zeros lie on $\T$.  There is a unimodular constant $\eta$ such that $\eta p$ is self-inversive.  Multiplication by $\eta$ does not change any of the quantities appearing in Theorem \ref{thm:main} or in \eqref{eq:sharp-intro}.
\end{lem}

\begin{proof}
Write
\[
p(z)=a\prod_{\nu=1}^n(z-\tau_\nu),\qquad |\tau_\nu|=1.
\]
Then $p^*=\lambda p$ for some unimodular $\lambda$.  Choose $\eta$ with $\eta^2=\lambda$.  Since $(\eta p)^*=\bar\eta p^*=\bar\eta\lambda p=\eta p$, the first assertion follows.  The remaining assertion is immediate because $|\eta p|=|p|$ and
\[
\eta p-\frac1nD(\eta p)=\eta\left(p-\frac1nDp\right).
\]
\end{proof}

From now on, until the approximation step in Section \ref{s:5}, we assume that $p=p^*$ and that $p$ has simple zeros on $\T$.  Write
\begin{equation}\label{eq:p-coeffs}
p(z)=\sum_{j=0}^n a_jz^j,
\qquad a_j=\overline{a_{n-j}}.
\end{equation}
Let $D=z\frac{d}{dz}$ and define
\begin{equation}\label{eq:q-def}
q=p-\frac1nDp.
\end{equation}
Then
\begin{equation}\label{eq:q-coeffs}
q(z)=\sum_{j=0}^{n-1}\frac{n-j}{n}a_jz^j.
\end{equation}
By \eqref{eq:p-coeffs}, and with the reflection still taken relative to degree $n$,
\begin{equation}\label{eq:q-star-Dp}
q^*(z)=\frac1nDp(z)=\sum_{j=1}^{n}\frac{j}{n}a_jz^j.
\end{equation}
Consequently,
\begin{equation}\label{eq:p-qqstar}
p=q+q^*.
\end{equation}

\begin{lem}\label{lem:q-zero-free}
Let $p$ have all zeros on $\T$ and let $q$ be given by \eqref{eq:q-def}.  Then $q$ has no zeros in $\D$.  If, in addition, the zeros of $p$ are simple and $p=p^*$, then $q$ has no zeros on $\overline{\D}$ and
$r=q^*/q$
is a finite Blaschke product satisfying $r(0)=0$.
\end{lem}

\begin{proof}
Let
\[
p(z)=a\prod_{\nu=1}^n(z-\tau_\nu),\qquad |\tau_\nu|=1.
\]
For $|z|<1$,
\[
\frac{Dp(z)}{p(z)}=\sum_{\nu=1}^n\frac{z}{z-\tau_\nu}.
\]
For $|\tau|=1$ one has
\begin{equation*}
\operatorname{Re}\frac{z}{z-\tau}
=
\frac12-\frac{1-|z|^2}{2|z-\tau|^2}<\frac12,
\qquad |z|<1.
\end{equation*}
Therefore
\[
\operatorname{Re}\frac{Dp(z)}{p(z)}<\frac n2,
\qquad |z|<1.
\]
If $q(z)=0$ in $\D$, then $Dp(z)/p(z)=n$, a contradiction.  Hence $q$ is zero free in $\D$.

Assume now that $p=p^*$ and that all zeros of $p$ on $\T$ are simple.  If $\zeta\in\T$ and $q(\zeta)=0$, then $q^*(\zeta)=\zeta^n\overline{q(\zeta)}=0$, and hence by \eqref{eq:p-qqstar}, $p(\zeta)=q(\zeta)+q^*(\zeta)=0$.  But then \eqref{eq:q-star-Dp} gives $Dp(\zeta)=0$, contradicting the simplicity of the zero $\zeta$ of $p$.  Thus $q$ has no zeros on $\overline\D$.

Since $|q^*|=|q|$ on $\T$, the quotient $r=q^*/q$ is analytic in $\D$, has unimodular boundary values, and is rational; hence it is a finite Blaschke product.  Finally $q$ has degree at most $n-1$, so the constant coefficient of $q^*$ is $0$, while $q(0)=a_0\ne0$.  Thus $r(0)=0$.
\end{proof}

\section{The Jensen term}\label{s:3}

Assume in this section that $p=p^*$ has simple zeros on $\T$, and keep the notation $q$ and $r=q^*/q$ from the previous section.  By \eqref{eq:p-qqstar},
\begin{equation}\label{eq:p-q-r}
p=q(1+r).
\end{equation}
Since $r(0)=0$ and $|r|=1$ almost everywhere on $\T$,
\begin{equation}\label{eq:one-plus-r-norm}
\int_{\T}|1+r|^2\dd=2.
\end{equation}
Let $N=N(p)=\int_{\T}|p|^2\dd$.

\begin{prop}\label{prop:jensen-term}
With the above notation,
\begin{equation}\label{eq:jensen-term}
\int_{\T}|p|^2\log |q|^2\dd
\ge
N\log\frac N2.
\end{equation}
\end{prop}

\begin{proof}
By \eqref{eq:one-plus-r-norm},
\[
d\mu=\frac12|1+r|^2\dd
\]
is a probability measure on $\T$.  From \eqref{eq:p-q-r},
\[
\int_{\T}|q|^2\,d\mu
=
\frac12\int_{\T}|q|^2|1+r|^2\dd
=
\frac N2.
\]
The function $\Phi(x)=x\log x$ is convex on $[0,\infty)$.  Therefore Jensen's inequality gives
\[
\frac12\int_{\T}|p|^2\log |q|^2\dd
=
\int_{\T}\Phi(|q|^2)\,d\mu
\ge
\Phi\left(\frac N2\right)
=
\frac N2\log\frac N2.
\]
This proves \eqref{eq:jensen-term}. 
\end{proof}

\section{The sharp polar derivative term}\label{s:4}

We now prove the sharp estimate for
\[
\int_{\T}|p|^2\log\left|\frac pq\right|^2\dd.
\]
The proof uses only elementary Hardy space notation.  We take
\[
\ip{f}{g}=\int_{\T} f\overline g\dd
\]
for the $H^2$ inner product.

\begin{lem}\label{lem:h-fourier}
For $|w|=1$, set
\[
h(w)=|1+w|^2\log |1+w|^2.
\]
Then the following Fourier series converges absolutely and uniformly on \(\T\):
\begin{equation}\label{eq:h-fourier}
h(w)
=
2+\frac32(w+\bar w)
+
\sum_{k=2}^{\infty}
\frac{2(-1)^k}{k(k^2-1)}(w^k+\bar w^k).
\end{equation}
In particular, the convergence also holds in \(L^1(\T)\).
\end{lem}

\begin{proof}
	For $0<\rho<1$, set
	\[
	h_\rho(w)=|1+w|^2\log |1+\rho w|^2 .
	\]
	Then
\begin{equation}\label{eq:log}
    	\log |1+\rho w|^2
	=
	\sum_{m=1}^{\infty}
	\frac{(-1)^{m+1}\rho^m}{m}(w^m+\bar w^m),
	\qquad |w|=1,
\end{equation}
	with absolute and uniform convergence.  Since as $\rho\to 1$,
	\[
	\log |1+\rho w|^2\longrightarrow \log |1+w|^2
	\qquad\hbox{in }L^1(\T),
	\]
	we have $h_\rho\to h$ in $L^1(\T)$.

    Let \(\widehat h(k)\) denote the \(k\)-th Fourier coefficient.  Since
\(h_\rho\to h\) in \(L^1(\T)\), for every fixed \(k\) we have
\[
        \widehat h(k)=\lim_{\rho\uparrow1}\widehat{h_\rho}(k).
\]
Write
\[
        \ell_m(\rho)=\frac{(-1)^{m+1}\rho^m}{m}\quad(m\ge1),
        \qquad \ell_0(\rho)=0.
\]
Multiplying the expansion in \eqref{eq:log} by
\[
        |1+w|^2=2+w+\bar w,
\]
and collecting the \(w^k\)-terms, we see that the coefficient of \(w^k\)
in \(h_\rho\) is
\[
        2\ell_k(\rho)+\ell_{k-1}(\rho)+\ell_{k+1}(\rho),
        \qquad k\ge1,
\]
and that the constant coefficient is \(2\ell_1(\rho)\).  Letting
\(\rho\uparrow1\) gives
\[
        \widehat h(0)=2,\qquad
        \widehat h(1)=\frac32,
\]
and, for \(k\ge2\),
\[
        \widehat h(k)=
        2\frac{(-1)^{k+1}}{k}
        +\frac{(-1)^k}{k-1}
        +\frac{(-1)^{k+2}}{k+1}
        =
        \frac{2(-1)^k}{k(k^2-1)}.
\]
The coefficients \(\widehat h(-k)\) are the same because \(h\) is real-valued.

	It remains only to justify the asserted mode of convergence.  Let \(H\) denote
the right-hand side of \eqref{eq:h-fourier}.  Since
\[
        \sum_{k=2}^{\infty}\frac{1}{k(k^2-1)}<\infty,
\]
the series defining \(H\) converges absolutely and uniformly on \(\T\).  Hence
\(H\) is continuous and its Fourier coefficients are exactly the coefficients
computed above.  Thus \(h-H\in L^1(\T)\) has all Fourier coefficients equal to
zero.  By the uniqueness theorem for Fourier coefficients in \(L^1(\T)\),
\(h=H\) almost everywhere.  With the convention \(h(-1)=0\), the function \(h\)
is continuous on \(\T\).  Therefore \(h=H\) everywhere on \(\T\), and the
Fourier series in \eqref{eq:h-fourier} converges absolutely and uniformly to
\(h\).
\end{proof}

For an analytic function $g(z)=\sum_{j\ge0}g_jz^j$, define the finite weighted quadratic form
\begin{equation*}
S_n(g)=\sum_{j=1}^{n-1}\frac{n-j}{j}|g_j|^2.
\end{equation*}
We establish the following weighted Schur contraction.
\begin{lem}\label{lem:weighted-schur}
Let $\varphi$ be a Schur function on $\D$, that is, $|\varphi|\le1$ in $\D$.  If $f\in H^2$ and $f(0)=0$, then
\begin{equation*}
S_n(\varphi f)\le S_n(f).
\end{equation*}
\end{lem}
\begin{proof}
	For $0<\rho<1$, put
	\[
	f_\rho(z)=f(\rho z),
	\qquad
	\varphi_\rho(z)=\varphi(\rho z).
	\]
 The Schur kernel
\[
        \frac{1-\varphi_\rho(z)\overline{\varphi_\rho(w)}}{1-z\bar w}
\]
is positive on \(\D\times\D\).  Since
\(f_\rho(z)\overline{f_\rho(w)}\) is a rank-one positive kernel, every
finite Gram matrix of
\[
K_\rho(z,w)=
\frac{
        f_\rho(z)\overline{f_\rho(w)}
        -
        \varphi_\rho(z)f_\rho(z)
        \overline{\varphi_\rho(w)f_\rho(w)}
}{1-z\bar w}
\]
is the Schur product of two positive semidefinite matrices.  By the Schur
product theorem \cite[p.~479, Theorem~7.5.3]{HJ13}, \(K_\rho\) is positive. Expand
\[
        K_\rho(z,w)=\sum_{i,j\ge0}c_{ij}^{(\rho)}z^i\bar w^j .
\]
Since \(K_\rho\) is positive, the coefficient matrix
\((c_{ij}^{(\rho)})_{i,j\ge0}\) is positive semidefinite; in particular,
\[
        c_{\ell\ell}^{(\rho)}\ge0,\qquad \ell\ge0.
\]
	
	For $g(z)=\sum_{j\ge0}g_jz^j$, the coefficient of $z^\ell\bar w^\ell$ in
	\[
	\frac{g(z)\overline{g(w)}}{1-z\bar w}
	\]
	is
	\[
	A_\ell(g)=\sum_{j=0}^{\ell}|g_j|^2 .
	\]
	Since the coefficient of $z^j$ in $f_\rho$ is $\rho^j f_j$, and the coefficient
	of $z^j$ in $\varphi_\rho f_\rho$ is $\rho^j(\varphi f)_j$, the inequality
	$c_{\ell\ell}^{(\rho)}\ge0$ gives
	\[
	\sum_{j=0}^{\ell}\rho^{2j}|(\varphi f)_j|^2
	\le
	\sum_{j=0}^{\ell}\rho^{2j}|f_j|^2 .
	\]
	Letting $\rho\uparrow1$, we obtain
	\begin{equation}\label{eq:partial-sums-contract}
		A_\ell(\varphi f)\le A_\ell(f),
		\qquad \ell\ge0.
	\end{equation}
	
	Set
	\[
	w_j=\frac{n-j}{j}\quad (1\le j\le n-1),
	\qquad
	w_n=0.
	\]
	Then $w_j-w_{j+1}>0$.  Since $f(0)=0$, also $(\varphi f)(0)=0$, and therefore
	\[
	S_n(g)=\sum_{\ell=1}^{n-1}(w_\ell-w_{\ell+1})A_\ell(g)
	\]
	for $g=f$ and $g=\varphi f$.  Combining this identity with
	\eqref{eq:partial-sums-contract} proves
	\[
	S_n(\varphi f)\le S_n(f).
	\]
\end{proof}

Next, we establish the following sharp polar derivative estimate.
\begin{thm}\label{thm:sharp-simple}
Let $n\ge2$.  Let $p=p^*$ be a degree $n$ polynomial with simple zeros on $\T$. Put
$
q=p-\frac1nDp.
$
Then
\begin{equation}\label{eq:sharp-strong}
\int_{\T}|p|^2\log\left|\frac pq\right|^2\dd
-
\int_{\T}|p|^2\dd
\ge
\frac{2\Gamma(p)}{n(n-1)},
\end{equation}
where
\begin{equation*}
\Gamma(p)=
\sum_{j=1}^{n-1}\frac{j(n-j)}{n^2}|a_j|^2
\end{equation*}
and $p(z)=\sum_{j=0}^na_jz^j$.  In particular,
\begin{equation}\label{eq:sharp-simple}
\int_{\T}|p|^2\log\left|\frac{p}{p-\frac1nDp}\right|^2\dd
\ge
\int_{\T}|p|^2\dd.
\end{equation}
\end{thm}

\begin{proof} By Lemma~\ref{lem:q-zero-free}, \(q\) has no zeros on \(\overline{\D}\),
and hence
$ r=q^*/q
$
is a finite Blaschke product with \(r(0)=0\).
 For $k\ge0$, define
\begin{equation*}
M_k=\int_{\T}|q|^2r^k\dd.
\end{equation*}
Since $r$ is inner, $M_0=\norm{q}_2^2$ and
\[
M_k=\ip{r^kq}{q}.
\]
By Parseval's identity, this inner product only involves the Taylor
coefficients of \(r^kq\) of degrees \(0,\ldots,n-1\), because \(q\) has
degree at most \(n-1\).  Since \(r(0)=0\), the function \(r^kq\) has a zero
of order at least \(k\) at the origin.  Hence
\begin{equation}\label{eq:Mk-vanish}
        M_k=0,\qquad k\ge n.
\end{equation}
By \eqref{eq:p-q-r},
\[
\int_{\T}|p|^2\log\left|\frac pq\right|^2\dd
=
\int_{\T}|q|^2 h(r)\dd.
\]
For \(L\ge2\), let
\[
s_L(w)
=
2+\frac32(w+\bar w)
+
\sum_{k=2}^{L}
\frac{2(-1)^k}{k(k^2-1)}(w^k+\bar w^k)
\]
be the \(L\)-th symmetric partial sum of the Fourier series in
Lemma~\ref{lem:h-fourier}.  By Lemma~\ref{lem:h-fourier}, \(s_L\to h\)
uniformly on \(\T\).  Since \(r\) has unimodular boundary values, this gives
\[
        \|s_L\circ r-h\circ r\|_{L^\infty(\T)}
        \le
        \|s_L-h\|_{L^\infty(\T)}
        \longrightarrow 0.
\]
As \(q\) is a polynomial, \(|q|^2\) is bounded on \(\T\).  Therefore
\[
        \int_{\T}|q|^2s_L(r)\dd
        \longrightarrow
        \int_{\T}|q|^2h(r)\dd .
\]
For each fixed \(L\), using the definition
\[
        M_k=\int_{\T}|q|^2r^k\dd,
\]
we have
\[
\int_{\T}|q|^2s_L(r)\dd
=
2M_0
+
3\operatorname{Re}M_1
+
4\sum_{k=2}^{L}
\frac{(-1)^k}{k(k^2-1)}\operatorname{Re}M_k .
\]
By \eqref{eq:Mk-vanish}, \(M_k=0\) for \(k\ge n\).  Passing to the limit
\(L\to\infty\), we obtain
\begin{equation}\label{eq:J-Mk}
\int_{\T}|p|^2\log\left|\frac pq\right|^2\dd
=
2M_0+3\operatorname{Re}M_1
+4\sum_{k=2}^{n-1}\frac{(-1)^k}{k(k^2-1)}\operatorname{Re}M_k.
\end{equation}

On the other hand,
\begin{equation}\label{eq:N-M0-M1}
\int_{\T}|p|^2\dd
=
\int_{\T}|q|^2|1+r|^2\dd
=
2M_0+2\operatorname{Re}M_1.
\end{equation}
We now estimate the moments $M_k$.

Let
\[
f_k=r^kq,
\qquad k\ge1.
\]
Then $f_1=q^*$, and, by \eqref{eq:q-coeffs} and \eqref{eq:q-star-Dp},
\[
(q)_j=\frac{n-j}{n}a_j,
\qquad
(f_1)_j=(q^*)_j=\frac{j}{n}a_j.
\]
Therefore
\begin{equation}\label{eq:M1-coeff}
M_1=\ip{f_1}{q}
=
\sum_{j=1}^{n-1}\frac{j(n-j)}{n^2}|a_j|^2
=\Gamma(p)
\ge0,
\end{equation}
and
\begin{equation*}
S_n(f_1)
=
\sum_{j=1}^{n-1}\frac{n-j}{j}\left|\frac{j}{n}a_j\right|^2
=\Gamma(p).
\end{equation*}
For $k\ge1$ we have $f_k=r^{k-1}f_1$.  Since $r^{k-1}$ is a Schur function and $f_1(0)=0$, Lemma \ref{lem:weighted-schur} yields
\begin{equation*}
S_n(f_k)\le S_n(f_1)=\Gamma(p).
\end{equation*}
By the weighted Cauchy-Schwarz inequality,
\begin{align}
|M_k|
&=|\ip{f_k}{q}|
=\left|\sum_{j=1}^{n-1}(f_k)_j\overline{q_j}\right| \notag\\
&\le
\left(\sum_{j=1}^{n-1}\frac{n-j}{j}|(f_k)_j|^2\right)^{1/2}
\left(\sum_{j=1}^{n-1}\frac{j}{n-j}|q_j|^2\right)^{1/2} \notag\\
&=S_n(f_k)^{1/2}\Gamma(p)^{1/2}
\le \Gamma(p).
\label{eq:Mk-bound}
\end{align}
Here the second equality uses the fact that $q$ has degree at most $n-1$ and $f_k$ has zero constant coefficient.

Subtracting \eqref{eq:N-M0-M1} from \eqref{eq:J-Mk} and using \eqref{eq:M1-coeff} gives
\[
\int_{\T}|p|^2\log\left|\frac pq\right|^2\dd
-
\int_{\T}|p|^2\dd
=
\Gamma(p)+4\sum_{k=2}^{n-1}\frac{(-1)^k}{k(k^2-1)}\operatorname{Re}M_k.
\]
By \eqref{eq:Mk-bound},
\begin{align*}
\int_{\T}|p|^2\log\left|\frac pq\right|^2\dd
-
\int_{\T}|p|^2\dd
&\ge
\Gamma(p)-4\Gamma(p)\sum_{k=2}^{n-1}\frac{1}{k(k^2-1)}.
\end{align*}
The elementary telescoping identity
\begin{equation*}
\sum_{k=2}^{n-1}\frac{1}{k(k^2-1)}
=
\frac14-\frac{1}{2n(n-1)}
\end{equation*}
therefore gives \eqref{eq:sharp-strong}.  The weaker inequality \eqref{eq:sharp-simple} follows immediately.
\end{proof}

\begin{rem}\label{rem:n-one}
For $n=1$, the same conclusion holds with equality.  Indeed, after self-inversive normalization, $p(z)=a_0+\bar a_0z$, $q=a_0$, and $r=q^*/q$ is a unimodular multiple of $z$.  Hence
\[
\int_{\T}|p|^2\log\left|\frac pq\right|^2\dd
=
|a_0|^2\int_{\T}h\dd
=
2|a_0|^2
=
\int_{\T}|p|^2\dd.
\]
\end{rem}

\section{Removing multiple zeros}\label{s:5}

The preceding argument was written for polynomials with simple zeros on $\T$, so that $q$ has no boundary zeros and $r=q^*/q$ is an ordinary finite Blaschke product without cancellations.  We now give the continuity and approximation details needed to let boundary zeros coalesce.

The first is a simple self-inversive approximation lemma.
\begin{lem}\label{lem:self-inversive-approx}
Let $p=p^*$ have degree $n$ and all zeros on $\T$.  Then there are degree $n$ polynomials $p_\nu=p_\nu^*$ with simple zeros on $\T$ such that $p_\nu\to p$ coefficientwise.
\end{lem}

\begin{proof}
Write
\[
p(z)=a\prod_{j=1}^n(z-\tau_j),\qquad |\tau_j|=1.
\]
Choose numbers $\varepsilon_{j,\nu}\to0$ so that, for each fixed $\nu$, the points
\[
\tau_{j,\nu}=\tau_j e^{i\varepsilon_{j,\nu}},
\qquad j=1,\dots,n,
\]
are pairwise distinct.  Put
\[
P_\nu(z)=a\prod_{j=1}^n(z-\tau_{j,\nu}).
\]
Then $P_\nu$ has simple zeros on $\T$ and $P_\nu\to p$ coefficientwise.  Since all zeros of $P_\nu$ lie on $\T$, there is a unimodular $\lambda_\nu$ such that $P_\nu^*=\lambda_\nu P_\nu$.  Because $P_\nu^*\to p^*=p$ and $P_\nu\to p\ne0$, every convergent subsequence of $\lambda_\nu$ has limit $1$; hence $\lambda_\nu\to1$.

For all sufficiently large \(\nu\), choose a square root \(\eta_\nu\) of
\(\lambda_\nu\) such that \(\eta_\nu\to1\), and define
\[
        p_\nu=\eta_\nu P_\nu .
\]
Then
\[
        p_\nu^*
        =
        \overline{\eta_\nu}P_\nu^*
        =
        \overline{\eta_\nu}\lambda_\nu P_\nu
        =
        \eta_\nu P_\nu
        =
        p_\nu .
\]
Moreover, \(p_\nu\) has simple zeros on \(\T\), has degree \(n\), and
\(p_\nu\to p\) coefficientwise.  Omitting finitely many initial indices
completes the proof.
\end{proof}
We also establish the following logarithmic continuity.
\begin{lem}\label{lem:log-continuity}
Let $A_\nu$ and $B_\nu$ be sequences of nonzero polynomials of uniformly bounded degree, and suppose that $A_\nu\to A$ and $B_\nu\to B$ coefficientwise, where $A$ and $B$ are nonzero.  Then
\begin{equation*}
\int_{\T}|A_\nu|^2\log |B_\nu|^2\dd
\longrightarrow
\int_{\T}|A|^2\log |B|^2\dd.
\end{equation*}
\end{lem}

\begin{proof}
	Let $d$ be a common upper bound for the degrees of $A_\nu,B_\nu,A,B$.  We first
	prove that
	\begin{equation}\label{eq:logB-L1}
		\log |B_\nu|\longrightarrow \log |B|
		\qquad\hbox{in }L^1(\T).
	\end{equation}
	
	We use the following elementary fact.  If
	\[
	(u_\nu,v_\nu)\to(u,v),
	\qquad |u_\nu|^2+|v_\nu|^2=|u|^2+|v|^2=1,
	\]
	then
	\begin{equation}\label{eq:linear-factor-L1}
		\log |u_\nu e^{it}-v_\nu|
		\longrightarrow
		\log |ue^{it}-v|
		\qquad\hbox{in }L^1([0,2\pi]).
	\end{equation}
	Indeed, if $u=0$, then $|v|=1$ and
	$u_\nu e^{it}-v_\nu\to -v$ uniformly, so the convergence is uniform.  If
	$u\ne0$, put
	\[
	a_\nu=\frac{v_\nu}{u_\nu},
	\qquad
	a=\frac vu .
	\]
	Then, for all large $\nu$,
	\[
	\log |u_\nu e^{it}-v_\nu|
	=
	\log |u_\nu|+\log |e^{it}-a_\nu|.
	\]
	If $|a|\ne1$, the convergence is again uniform.  It remains to consider the
	case $|a|=1$.  After a rotation, assume $a=1$.  For $a'$ sufficiently close to
	$1$, write
	\[
	a'=re^{i\phi},
	\qquad
	\frac12\le r\le2.
	\]
	For $|s|\le\pi$,
	\[
	|e^{is}-r|^2=(1-r)^2+2r(1-\cos s)\ge c s^2
	\]
	with an absolute constant $c>0$.  Hence, with
	$\log^- x=\max\{-\log x,0\}$,
	\[
	\log^- |e^{it}-a'|
	\le
	C+\log^- |t-\phi|
	\]
	for a constant $C$ independent of $a'$ near $1$.  Also
	$\log^+ |e^{it}-a'|$ is uniformly bounded.  Therefore, for every arc $I$ of
	length $|I|\le1$,
	\begin{equation}\label{eq:uniform-log-integrability}
		\sup_{a'\ \mathrm{near}\ 1}
		\int_I |\log |e^{it}-a'||\,dt
		\le
		C |I|\log\frac{e}{|I|}.
	\end{equation}
Given \(\varepsilon>0\), choose an arc \(I\subset[-\pi,\pi]\) centered at
\(0\) so small that the right-hand side of
\eqref{eq:uniform-log-integrability} is less than \(\varepsilon\).  For
all sufficiently large \(\nu\), the argument of \(a_\nu\) lies in the arc
with the same center and half the length of \(I\).  On
\([-\pi,\pi]\setminus I\), the functions
\(\log |e^{it}-a_\nu|\) converge uniformly to \(\log |e^{it}-1|\).
This proves \eqref{eq:linear-factor-L1}.
	
	We now prove \eqref{eq:logB-L1}.  It is enough to show that every subsequence of
	$\log |B_\nu|$ has a further subsequence converging to $\log |B|$ in
	$L^1(\T)$.  Choose an arbitrary subsequence, still denoted by $B_\nu$.  For each
	$\nu$, write the homogeneous factorization
	\[
	B_\nu(z)
	=
	C_\nu\prod_{j=1}^{d}(u_{j,\nu}z-v_{j,\nu}),
	\qquad
	|u_{j,\nu}|^2+|v_{j,\nu}|^2=1.
	\]
	Factors with $u_{j,\nu}=0$ represent roots at infinity.  Since the normalized
	products
	\[
	\prod_{j=1}^{d}(u_jz-v_j)
	\]
	form a compact family of nonzero polynomials, their coefficient vectors are
	bounded away from zero.  As the coefficient vectors of $B_\nu$ converge, the
	constants $C_\nu$ are bounded.  Passing to a further subsequence, we may assume
	that
	\[
	C_\nu\to C,
	\qquad
	(u_{j,\nu},v_{j,\nu})\to(u_j,v_j),
	\quad j=1,\dots,d.
	\]
	The coefficientwise limit is
	\[
	C\prod_{j=1}^{d}(u_jz-v_j).
	\]
	Since $B_\nu\to B\not\equiv0$, this limit equals $B$, and necessarily
	$C\ne0$.  By \eqref{eq:linear-factor-L1},
	\[
	\log |u_{j,\nu}e^{it}-v_{j,\nu}|
	\longrightarrow
	\log |u_je^{it}-v_j|
	\qquad\hbox{in }L^1
	\]
	for each $j$.  Hence
	\[
	\log |B_\nu|
	=
	\log |C_\nu|
	+
	\sum_{j=1}^{d}\log |u_{j,\nu}e^{it}-v_{j,\nu}|
	\longrightarrow
	\log |B|
	\qquad\hbox{in }L^1(\T)
	\]
	along this further subsequence.  Since every subsequence has a further
	subsequence with the same $L^1$ limit, \eqref{eq:logB-L1} follows for the whole
	sequence.
	
	Therefore
	\[
	\log |B_\nu|^2\longrightarrow \log |B|^2
	\qquad\hbox{in }L^1(\T).
	\]
	In particular, after discarding finitely many initial terms,
	\[
	\sup_\nu \norm{\log |B_\nu|^2}_{L^1(\T)}<\infty.
	\]
	Since $A_\nu\to A$ coefficientwise and the degrees are uniformly bounded,
	\[
	|A_\nu|^2\longrightarrow |A|^2
	\qquad\hbox{uniformly on }\T.
	\]
	Thus
	\begin{align*}
		&\left|
		\int_{\T}|A_\nu|^2\log |B_\nu|^2\dd
		-
		\int_{\T}|A|^2\log |B|^2\dd
		\right| \\
		&\le
		\norm{|A_\nu|^2-|A|^2}_{L^\infty(\T)}
		\norm{\log |B_\nu|^2}_{L^1(\T)} \\
		&\quad+
		\norm{|A|^2}_{L^\infty(\T)}
		\norm{\log |B_\nu|^2-\log |B|^2}_{L^1(\T)}.
	\end{align*}
	The right-hand side tends to $0$, proving
	\[
	\int_{\T}|A_\nu|^2\log |B_\nu|^2\dd
	\longrightarrow
	\int_{\T}|A|^2\log |B|^2\dd .
	\]
\end{proof}
We next define the ratio functional.
\begin{defin}\label{def:ratio-functional}
	Let $P$ be a degree $n$ polynomial with all zeros on $\T$, and set
	\[
	Q=P-\frac1nDP.
	\]
	Since $Q(0)=P(0)\ne0$, $Q$ is not the zero polynomial.  We define
	\begin{equation}\label{eq:ratio-functional-def}
		\mathcal J_n(P)=
		\int_{\T}|P|^2\log|P|^2\dd
		-
		\int_{\T}|P|^2\log|Q|^2\dd.
	\end{equation}
	When the pointwise quotient is unambiguous, this is equal to
	\[
	\int_{\T}|P|^2\log\left|\frac{P}{Q}\right|^2\dd.
	\]
	If $P$ and $Q$ have common boundary zeros, the pointwise quotient expression may
	be ambiguous at those points.  In that case, throughout the paper the notation
	\[
	\int_{\T}|P|^2\log\left|\frac{P}{Q}\right|^2\dd
	\]
	means the quantity $\mathcal J_n(P)$ defined in
	\eqref{eq:ratio-functional-def}.  The two logarithmic integrals in
	\eqref{eq:ratio-functional-def} are finite by Lemma
	\ref{lem:log-continuity}, applied to fixed nonzero polynomials.
\end{defin}
Next, we estimate the Jensen term with multiple zeros.
\begin{cor}\label{cor:jensen-general}
Let $p=p^*$ have degree $n$ and all zeros on $\T$, not necessarily simple.  Put
\[
q=p-\frac1nDp,
\qquad N(p)=\int_{\T}|p|^2\dd.
\]
Then
\begin{equation}\label{eq:jensen-general}
\int_{\T}|p|^2\log|q|^2\dd
\ge
N(p)\log\frac{N(p)}2.
\end{equation}
\end{cor}

\begin{proof}
Choose $p_\nu=p_\nu^*$ as in Lemma \ref{lem:self-inversive-approx}, and set
\[
q_\nu=p_\nu-\frac1nDp_\nu.
\]
For each $\nu$, Proposition \ref{prop:jensen-term} gives
\[
\int_{\T}|p_\nu|^2\log|q_\nu|^2\dd
\ge
N(p_\nu)\log\frac{N(p_\nu)}2.
\]
Since $q_\nu\to q$ coefficientwise and $q(0)=p(0)\ne0$, Lemma \ref{lem:log-continuity} gives convergence of the left side, while $N(p_\nu)\to N(p)$.  Passing to the limit proves \eqref{eq:jensen-general}.
\end{proof}
We also need the sharp polar derivative estimate with multiple zeros.
\begin{cor}\label{cor:sharp-general}
Let $p=p^*$ have degree $n\ge2$ and all zeros on $\T$, not necessarily simple.  Let
\[
q=p-\frac1nDp,
\qquad p(z)=\sum_{j=0}^na_jz^j.
\]
Then, with the quotient expression understood as in Definition \ref{def:ratio-functional},
\begin{equation}\label{eq:sharp-general-strong}
\int_{\T}|p|^2\log\left|\frac pq\right|^2\dd
-
\int_{\T}|p|^2\dd
\ge
\frac{2}{n(n-1)}
\sum_{j=1}^{n-1}\frac{j(n-j)}{n^2}|a_j|^2.
\end{equation}
\end{cor}

\begin{proof}
Choose $p_\nu=p_\nu^*$ as in Lemma \ref{lem:self-inversive-approx}.  Theorem \ref{thm:sharp-simple} applies to $p_\nu$.  By Lemma \ref{lem:log-continuity}, both terms in the definition of $\mathcal J_n(p_\nu)$ converge to the corresponding terms for $p$, and $\int|p_\nu|^2\dd\to\int|p|^2\dd$.  The coefficient expression on the right side of \eqref{eq:sharp-general-strong} is continuous.  Passing to the limit proves the result.
\end{proof}

\section{Proof of Theorem \ref{thm:main}}\label{s:6}

We now prove Theorem \ref{thm:main}. 
\begin{proof}[Proof of Theorem \ref{thm:main}]
By Lemma \ref{lem:self-inversive-normalization}, we may multiply $p$ by a unimodular constant and assume $p=p^*$; this changes neither $|p|$, nor $\left|p-\frac1nDp\right|$, nor any coefficient modulus appearing below.  Put
\[
q=p-\frac1nDp,
\qquad
N=N(p).
\]

First assume $n=1$.  Then every degree-one polynomial with its zero on $\T$ has the form $c(\omega+z)$ with $|\omega|=1$, and the equality computation at the end of the proof gives the asserted statement.  Hence assume $n\ge2$.

By Corollary \ref{cor:jensen-general},
\[
\int_{\T}|p|^2\log |q|^2\dd
\ge
N\log\frac N2.
\]
By Corollary \ref{cor:sharp-general},
\[
\int_{\T}|p|^2\log\left|\frac pq\right|^2\dd
\ge
N+
\frac{2}{n(n-1)}
\sum_{j=1}^{n-1}\frac{j(n-j)}{n^2}|a_j|^2.
\]
The quotient integral is understood in the sense of
Definition~\ref{def:ratio-functional}; hence adding the two estimates gives the strengthened inequality
\begin{equation}\label{eq:strong-main}
\int_{\T}|p|^2\log |p|^2\dd
\ge
N\left(1+\log\frac N2\right)
+
\frac{2}{n(n-1)}
\sum_{j=1}^{n-1}\frac{j(n-j)}{n^2}|a_j|^2.
\end{equation}
In particular, \eqref{eq:main-homogeneous} follows.

It remains to identify equality.  Suppose equality holds in \eqref{eq:main-homogeneous}.  Apply the preceding normalization if necessary.  Since multiplication by a unimodular constant multiplies every coefficient by that same constant, it does not change which coefficients vanish.  The non-negative remainder term in \eqref{eq:strong-main} must be zero.  Thus
\[
a_1=a_2=\cdots=a_{n-1}=0.
\]
Hence
\[
p(z)=a_0+a_nz^n.
\]
Since all zeros of $p$ lie on $\T$, one must have $|a_0|=|a_n|$.  Therefore $p$ has the form
\[
p(z)=c(\omega+z^n),
\qquad |\omega|=1,
\]
with $c\ne0$.

Conversely, if $p(z)=c(\omega+z^n)$ with $|\omega|=1$, then
\[
N(p)=2|c|^2.
\]
The unimodular constant $\omega$ is removed by a rotation of the variable.  Thus, writing $z=e^{i\theta}$ and changing variables,
\begin{align*}
\int_{\T}|p|^2\log |p|^2\dd
&=
|c|^2\int_{\T}|1+z^n|^2
\log\bigl(|c|^2|1+z^n|^2\bigr)\dd\\
&=
N(p)\log |c|^2
+
|c|^2\int_{\T}|1+z|^2\log |1+z|^2\dd.
\end{align*}
By Lemma \ref{lem:h-fourier}, the last integral is the constant Fourier coefficient of $h$, namely $2$.  Thus
\[
\int_{\T}|p|^2\log |p|^2\dd
=
2|c|^2(1+\log |c|^2)
=
N(p)\left(1+\log\frac{N(p)}2\right),
\]
which proves equality for precisely the stated family.  The theorem follows.
\end{proof}
%



\begin{thebibliography}{99}






\bibitem{AM21}
J.~Agler and J.~E. McCarthy, \emph{The Krzy\.z conjecture and an entropy conjecture},
J. Anal. Math. \textbf{144} (2021), 207--226. \doi{10.1007/s11854-021-0178-z}.




\bibitem{BGI23}
O.~F. Brevig, S.~Grepstad, and S.~M. F. Instanes, \emph{F. Wiener's trick and an
extremal problem for $H^p$}, Comput. Methods Funct. Theory \textbf{23} (2023),
697--722. \doi{10.1007/s40315-022-00469-x}.



\bibitem{Bro87}
J.~E. Brown, \emph{Iteration of functions subordinate to schlicht functions},
Complex Var. Theory Appl. \textbf{9} (1987), no.~2--3, 143--152.
\doi{10.1080/17476938708814258}.

\bibitem{Erm90}
R.~J.~P.~M. Ermers, \emph{Coefficient estimates for bounded non-vanishing
functions}, Ph.D. thesis, Katholieke Universiteit Nijmegen, Wibro
Dissertatiedrukkerij, Helmond, 1990.

\bibitem{HJ13}
R.~A. Horn and C.~R. Johnson, \emph{Matrix Analysis}, 2nd ed.,
Cambridge University Press, Cambridge, 2013.

\bibitem{Hor78}
Ch.~Horowitz, \emph{Coefficients of nonvanishing functions in $H^\infty$},
Israel J. Math. \textbf{30} (1978), no.~3, 285--291. \doi{10.1007/BF02761076}.

\bibitem{HSZ77}
J.~A. Hummel, S.~Scheinberg, and L.~Zalcman, \emph{A coefficient problem for
bounded nonvanishing functions}, J. Anal. Math. \textbf{31} (1977), 169--190.
\doi{10.1007/BF02813302}.


\bibitem{Krz68}
J.~G. Krzy\.z, \emph{Coefficient problem for bounded non-vanishing functions},
Ann. Polon. Math. \textbf{20} (1968), 314.







\bibitem{MSUV15}
M.~J. Mart\'in, E.~T. Sawyer, I.~Uriarte-Tuero, and D.~Vukoti\'c,
\emph{The Krzy\.z conjecture revisited}, Adv. Math. \textbf{273} (2015),
716--745. \doi{10.1016/j.aim.2014.12.031}.




\bibitem{Sam03}
N.~Samaris, \emph{A proof of Krzy\.z's conjecture for the fifth coefficient},
Complex Var. Theory Appl. \textbf{48} (2003), no.~9, 753--766.
\doi{10.1080/0278107031000152616}.



\bibitem{Tan83}
D.~L. Tan, \emph{Coefficient estimates for bounded nonvanishing functions},
Chinese Ann. Math. Ser. A \textbf{4} (1983), no.~1, 97--104; English summary in
Chinese Ann. Math. Ser. B \textbf{4} (1983), no.~1, 131--132.


\end{thebibliography}
\end{document}